\documentclass[12pt]{amsart}
\usepackage{amscd,amsmath,amssymb,amsfonts}
\usepackage{url,enumerate}
\usepackage{dsfont}
\usepackage[cmtip, all]{xy}
\theoremstyle{plain}
\newtheorem{thm}{Theorem}
\newtheorem{lem}[thm]{Lemma}
\newtheorem{cor}[thm]{Corollary}

\theoremstyle{definition}

\newtheorem{rmk}[thm]{Remark}

\numberwithin{thm}{section}
\numberwithin{equation}{thm}

\newcommand{\Hom}{{\rm Hom}}

\newcommand{\Trace}{{\rm Trace}}
\newcommand{\sym}{\mathrm{Sym}}
\newcommand{\alt}{\mathrm{Alt}}
\newcommand{\Aut}{\mathrm{Aut}}
\newcommand{\Out}{\mathrm{Out}}
\newcommand{\Irr}{\mathrm{Irr}}
\newcommand{\Innd}{\mathrm{Inndiag}}
\newcommand{\eps}{\epsilon}

\newcommand{\SL}{\mathrm{SL}}
\newcommand{\PSL}{\mathrm{PSL}}
\newcommand{\GL}{\mathrm{GL}}
\newcommand{\PGL}{\mathrm{PGL}}
\newcommand{\SU}{\mathrm{SU}}
\newcommand{\PSU}{\mathrm{PSU}}
\newcommand{\Sp}{\mathrm{Sp}}
\newcommand{\PSp}{\mathrm{PSp}}
\newcommand{\SO}{\mathrm{SO}}

\newcommand{\sgn}{\mathrm{sgn}}

\newcommand{\Norm}{{\rm Norm}}

\newcommand{\sC}{{\mathcal C}}

\newcommand{\sF}{{\mathcal F}}
\newcommand{\sG}{{\mathcal G}}
\newcommand{\sH}{{\mathcal H}}

\newcommand{\sL}{{\mathcal L}}

\newcommand{\sP}{{\mathcal P}}

\newcommand{\sR}{{\mathcal R}}

\newcommand{\sT}{{\mathcal T}}

\newcommand{\A}{{\mathbb A}}

\newcommand{\C}{{\mathbb C}}
{

\newcommand{\F}{{\mathbb F}}
\newcommand{\G}{{\mathbb G}}

\newcommand{\Q}{{\mathbb Q}}

\newcommand{\Z}{{\mathbb Z}}

\newcommand{\triv}{{\mathds{1}}}
\newcommand{\tw}[1]{{}^#1\!}

\begin{document}
\title{Rigid Local Systems and Alternating Groups}
\author{Robert M. Guralnick}
\address{Department of Mathematics, University of Southern California,
Los Angeles, CA 90089-2532, USA}
\email{guralnic@usc.edu}
\author{Nicholas M. Katz}
\address{Department of Mathematics, Princeton University, Princeton, NJ 08544-1000, USA}
\email{nmk@math.princeton.edu}
\author{Pham Huu Tiep}
\address{Department of Mathematics, Rutgers University, Piscataway, NJ 08854-8019, USA}
\email{tiep@math.rutgers.edu}

\thanks{The first author was partially supported by NSF grant DMS-1600056 and
the third author was partially supported by NSF grant DMS-1665014.   The first author
would also like to thank the Institute for Advanced Study, Princeton for its support}.

\maketitle

\tableofcontents

\section{Introduction}

In earlier work \cite{Ka-RLSFM}, Katz exhibited some very simple one parameter families of exponential sums which gave rigid local systems on the affine line in characteristic $p$ whose geometric (and usually, arithmetic) monodromy groups were $\SL_2(q)$, and he exhibited other such very simple families giving $\SU_3(q)$. [Here $q$ is a power of the characteristic $p$, and $p$ is odd.] In this paper, we exhibit equally simple families whose geometric monodromy groups are the alternating groups $\alt(2q)$. We also determine their arithmetic monodromy groups. See Theorem 3.1. [Of course from the resolution \cite{Ray} of the Abhyankar Conjecture, any finite simple group whose order is divisible by $p$ will
occur as the geometric monodromy group of some local system on $\A^1/\overline{\F}_p$; the interest here is that it occurs in our particularly simple local systems.]

In the earlier work of Katz, he used a theorem to Kubert to know that the monodromy groups in question were finite, then work of Gross \cite{Gross} to determine which finite groups they were.
Here we do not have, at present, any direct way of showing this finiteness. Rather, the situation is more complicated and more interesting. Using some basic information about these local systems (cf. Theorem 6.1), the first and third authors prove a fundamental dichotomy:  the geometric monodromy group is either $\alt(2q)$ or it is the special orthogonal group $\SO(2q-1)$. The second author uses an elementary polynomial identity to compute the third moment as being $1$ (cf. Theorem 7.1), which rules out the
$\SO(2q-1)$ case. This roundabout method establishes the theorem. It would be interesting to find a ``direct" proof that
these local systems have integer (rather than rational) traces; this integrality is in fact equivalent to their monodromy groups being finite, cf. \cite [8.14.6]{Ka-ESDE}. But even if one had such a direct proof, it would still require serious group theory to show that their
geometric monodromy groups are the alternating groups.

\section{The local systems in general}
Throughout this paper, $p$ is an odd prime, $q$ is a power of $p$, $k$ is a finite field of charactertistic $p$, $\ell$ is a prime $\neq p$,  
$$\psi = \psi_k: (k, +) \rightarrow \mu_p \subset \overline{\Q_\ell}^\times$$
 is a nontrivial additive character of $k$, and
 $$\chi_2=\chi_{2,k}:k^\times \rightarrow \pm 1\subset \overline{\Q_\ell}^\times$$
 is the quadratic character, extended to $k$ by $\chi_2(0):=0$.
 For $L/k$ a finite extension, we have the nontrivial additive character
 $$\psi_{L/k}:=\psi_k\circ \Trace_{L/k}$$
 of $L$, and the quadratic character $\chi_{2,L}=\chi_{2,k}\circ \Norm_{L/k}$ of $L^\times$, extended to $L$ by 
 $\chi_{2,L}(0)=0$.
 
 On the affine line $\A^1/k$, we have the Artin-Schreier sheaf $\sL_{\psi(x)}$. On $\G_m/k$ we have the Kummer sheaf $\sL_{\chi_{2}(x)}$ and its extension by zero $j_!\sL_{\chi_{2}(x)}$ (for $j:\G_m \subset \A^1$ the inclusion) on $\A^1/k$.
 
 For an odd integer $n=2d+1$ which is prime to $p$, we have the rigid local system
 $$\sF(k,n,\psi) :=FT_\psi(\sL_{\psi(x^n)}\otimes j_!\sL_{\chi_{2}(x)})$$
 on $\A^1/k$. Let us recall the basic facts about it, cf. \cite[1.3 and 1.4]{Ka-NG2}.
 
  It is lisse of rank $n$, pure of weight one, and orthogonally self dual, with its geometric monodromy group
 $$G_{geom} \subset \SO(n,\overline{\Q_\ell}).$$
Recall that $G_{geom}$ is the Zariski closure in $\SO(n,\overline{\Q_\ell})$ of the image of the geometric fundamental group $\pi_1(\A^1/\overline{k})$ in the representation which ``is" the local system $\sF(k,n,\psi)$. For ease of later reference, we recall the following fundamental fact.
\begin{lem}\label{dense}For any lisse local system on $\A^1/\overline{k}$, the subgroup $\Gamma_p$ of its $G_{geom}$ generated by elements of $p$-power order is Zariski dense.
\end{lem}
\begin{proof}Denote by $N$ the Zariski closure of $\Gamma_p$. Then $N$ is a normal subgroup of $G_{geom}$.
We must show that the quotient $M:=G_{geom}/N$ is trivial. We first note that $M/M^0$ is trivial, as it is a finite, prime to $p$ quotient of $\pi_1(\A^1/\overline{k})$. Thus $M$ is connected.
We next show that $M^{red}:=M/\sR_u$, the quotient of $M$ by its unipotent radical, is trivial. For this, it suffices to show
that $M$ has no nontrivial irreducible representations. But any such representation is a local system on $\A^1/\overline{k}$ which is tamely ramified at $\infty$, so is trivial. Thus $M$ is unipotent. But $H^1(\A^1/\overline{k},\overline{\Q_\ell})$ vanishes, so any unipotent local system on  $\A^1/\overline{k}$ is trivial, and hence $M$ is trivial.
\end{proof}

Let us denote by $A(k,n,\psi)$ the Gauss sum
 $$A(k,n,\psi):= -\chi_2(n(-1)^d)\sum_{x \in k^\times}\psi(x)\chi_2(x).$$
 By the Hasse-Davenport relation, for $L/k$ an extension of degree $d$, we have
 $$A(L,n,\psi_{L/k})=(A(k,n,\psi))^d.$$
The twisted local system
 $$\sG(k,n,\psi):= \sF(k,n,\psi)\otimes A(n,k,\psi)^{-deg}$$
 is pure of weight zero and has
 $$G_{geom}\subset G_{arith} \subset  \SO(n,\overline{\Q_\ell}).$$
 
 Concretely, for $L/k$ a finite extension, and $t \in L$, the trace at time $t$ of $\sG(k,n,\psi)$ is
 $$\Trace(Frob_{t,L}|\sG(k,n,\psi) )= -(1/A(L,n,\psi_{L/k}))\sum_{x \in L^\times}\psi_{L/k}(x^n+tx)\chi_{2,L}(x)=$$
 $$=-(1/A(L,n,\psi_{L/k}))\sum_{x \in L}\psi_{L/k}(x^n+tx)\chi_{2,L}(x),$$
 the last equality because the $\chi_2$ factor kills the $x=0$ term.

 Let us recall also \cite[3.4]{Ka-NG2} that the geometric monodromy group of $\sF(k,n,\psi)$, or equivalently of $\sG(k,n,\psi)$, is independent of the choice of the pair $(k, \psi)$.
 
 To end this section, let us recall the relation of the local system $\sF(k,n,\psi)$ to the hypergeometric sheaf 
 $$\sH_n:= \sH(!,\psi;{\rm all\ char's\ of\  order\  dividing\ }n;\chi_2).$$
 According to \cite[9.2.2]{Ka-ESDE}, $\sF(k,n,\psi)|\G_m$ is geometrically isomorphic to a multiplicative translate of the Kummer pullback $[n]^\star \sH_n$. [An explicit descent of $\sF(k,n,\psi)|\G_m$ through the $n$'th power map
 is given by the lisse sheaf on $\G_m$ whose trace function at time $t \in L^\times$, for $L/k$ a finite extension, is
 $$t \mapsto -\sum_{x \in L^\times}\psi_{L/k}(x^n/t + x)\chi_{2,L}(x/t).$$
 The structure theory of hypergeometric sheaves show that this descent is, geometrically, a multiplicative translate of the asserted $\sH_n$.]
 
 \section{The candidate local systems for $\alt(2q)$}
 In this section, we specialize the $n$ of the previous section to
 $$n=2q-1 = 2(q-1) +1.$$
 
 The target theorem is this.
 \begin{thm}Let $p$ be an odd prime, $q$ a power of $p$, $k$ a finite field of characteristic $p$, $\ell$ a prime $\neq p$,
 and $\psi$ a nontrivial additive character of $k$. For the $\ell$-adic local system $\sG(k,2q-1,\psi)$ on $\A^1/k$, its geometric and arithmetic monodromy groups are given as follows.
 \begin{itemize}
 \item[(1)]$G_{geom} =\alt(2q)$ in its unique irreducible representation of dimension $2q-1$.
 \item[(2a)]If $-1$ is a square in $k$, then $G_{geom}=G_{arith}=\alt(2q)$.
  \item[(2b)]If $-1$ is not a square in $k$, then $G_{arith}= \sym(2q)$, the symmetric group, in its irreducible representation labeled by the partition $(2,1^{2q-2})$, i.e. \newline
  ${\rm(the\ deleted\ permutation\ representation\ of\  }\sym(2q)) \otimes \sgn.$
 \end{itemize}
  \end{thm}
  
  \begin{rmk}The traces of elements of $\alt(n)$ (respectively of $\sym(n)$) in its deleted permutation representation (respectively in every irreducible representation) are integers. One sees easily (look at the action of $Gal(\Q(\zeta_p)/\Q)$) that the local system $\sG(k,2q-1,\psi)$ has traces which all lie in $\Q$,
  but as mentioned in the introduction, we do not know a direct proof that these traces all lie in $\Z$.
  \end{rmk} 
  
 \section{Basic facts about $\sH_n$}
 In this section, we assume that $n \ge 3$ is odd and that $n(n-1)$ is prime to $p$. The geometric local monodromy at $0$ is tame, and a topological generator of the tame inertia group $I(0)^{tame}$, acting on $\sH_n$, has as eigenvalues all the roots of unity of order dividing $n$. 
 
 The geometric local monodromy at $\infty$ is the direct sum
 $$\sL_{\chi_2}\oplus W; \ \ W{\rm\  has\  rank\ }n-1,\ {\rm and\  all\ slopes\ }1/(n-1).$$

Because $n$ is odd, the local system $\sH_n$ is (geometrically) orthogonally self dual, and $\det(\sH_n)$ is geometrically trivial (because trivial at $0$, lisse on $\G_m$, and all $\infty$ slopes are $\le 1/(n-1) < 1$). Therefore $\det(W)$ is geometrically $\sL_{\chi_2}$. From \cite[8.6.4 and 8.7.2]{Ka-ESDE}, we see that up to multiplicative translation, the
geometric isomorphism class is determined entirely by its rank $n-1$ and its determinant $\sL_{\chi_2}$. Because $n-1$ is even and prime to $p$, it follows that up to multiplicative translation, the
geometric isomorphism class of $W$ is that of the $I(\infty)$-representation of the Kloosterman sheaf
$$Kl_{n-1}:=Kl(\psi;{\rm all\ char's\ of\  order\  dividing\ }n-1).$$

By \cite[5.6.1]{Ka-GKM}, we have a global Kummer direct image geometric isomorphism 
$$Kl_{n-1} \cong [n-1]_\star \sL_{\psi_{n-1}},$$
where we write $\psi_{n-1}$ for the additive character $x \mapsto \psi((n-1)x)$.
Therefore, up to multiplicative translation, the
geometric isomorphism class of $W$ is that of $ [n-1]_\star \sL_\psi$.
Pulling back by $[n-1]$, which does not change the restriction of $W$ to the wild inertia group $P(\infty)$, we get
$$[n-1]^\star W \cong \oplus_{\zeta \in \mu_{n-1}}\sL_{\psi(\zeta x)}.$$

A further pullback by $n$'th power, which also does not change the restriction of $W$ to $P(\infty)$, gives
$$[n-1]^\star [n]^\star W \cong \oplus_{\zeta \in \mu_{n-1}}\sL_{\psi(\zeta x^n)}.$$
Thus we find that the $I(\infty)$ representation attached to a multiplicative translate of 
 $[n-1]^\star \sF(k,n,\psi)$ is the direct sum
 $$\triv \bigoplus _{\zeta \in \mu_{n-1}}\sL_{\psi(\zeta x^n)}= \bigoplus _{\alpha \in \mu_{n-1}\cup \{0\}}\sL_{\psi(\alpha x^n)}.$$
 
This description shows that the image of $P(\infty)$ in the $I(\infty)$-representation attached to $\sF(k,n,\psi)$ is an abelian group killed by $p$.
 
\begin{lem}Let $L/\F_p$ be a finite extension which contains the $n-1$'st roots of unity. Denote by $V \subset L$
the additive subgroup of $L$ spanned by the $n-1$'st roots of unity. Denote by
$V^\star$ the Poincar\'{e} dual of $V$:
$$V^\star := \Hom_{\F_p}(V, \mu_p(\overline{\Q_\ell})).$$
Then the image of $P(\infty)$ in the $I(\infty)$-representation attached to $\sF(k,n,\psi)$ is $V^\star$, and the representation restricted to $V^\star$ is the direct sum of
$$\triv \bigoplus_{\zeta \in \mu_{n-1}(L)} ({\rm evaluation\ at\ }\zeta)=\bigoplus_{\alpha \in \mu_{n-1}(L)\cup \{0\}} ({\rm evaluation\ at\ }\alpha).$$
\end{lem}
\begin{proof}Each of the characters $\sL_{\psi(\alpha x^n)}$ of $I(\infty)$ has order dividing $p$. Given an $n$-tuple
of elements $(a_\alpha)_{\alpha \in \mu_{n-1}(L)\cup \{0\}}$, consider the character
$$\Lambda :=\otimes_{\alpha \in \mu_{n-1}(L)\cup \{0\}}(\sL_{\psi(\alpha x^n)})^{\otimes a_\alpha}=$$
$$=\sL_{\psi((\sum_{\alpha \in \mu_{n-1}(L)\cup \{0\}}a_\alpha \alpha)x^n)}.$$
The following conditions are equivalent.
\begin{itemize}
\item[(1)]$\sum_{\alpha \in \mu_{n-1}(L)\cup \{0\}}a_\alpha \alpha=0$.
\item[(2)]The character $\Lambda$ is trivial on $I(\infty)$.
\item[(3)]The character $\Lambda$ is trivial on $P(\infty)$.
\end{itemize}
Indeed, it is obvious that $(1)\implies (2) \implies (3)$. If (3) holds, then for $A:=\sum_{\alpha \in \mu_{n-1}(L)\cup \{0\}}a_\alpha \alpha$, we have that $\sL_{\psi(Ax)}$ is trivial on $P(\infty)$, so is a character of $I(\infty)/P(\infty) = I(\infty)^{tame}$, a group of order prime to $p$. But $\sL_{\psi(Ax)}$ has order dividing $p$, so is trivial on $I(\infty)$, hence $A=0$.

This equivalence shows that the character group of the image of $P(\infty)$ is indeed the $\F_p$ span of the $\alpha$'s, i.e., it is $V$. The rest is just Poincar\'{e} duality of finite abelian groups.
\end{proof}

\section {basic facts about $\sH_{2q-1}$}
Taking $n=2q-1$, the geometric local monodromy at $0$ of $\sH_{2q-1}$ is tame, and a topological generator of the tame inertia group $I(0)^{tame}$, acting on $\sH_n$, has as eigenvalues all the roots of unity of order dividing $2q-1$. 

Turning now to the action of $P(\infty)$, we have
\begin{lem}Denote by $\zeta_{2q-2} \in \F_{q^2}$ a primitive $2q-2$'th root of unity. In the $I(\infty)$-representation attached to $\sF(k,2q-1,\psi)$, the character group $V$ of the image of $P(\infty)$ is the $\F_p$-space
$$V= \F_q \oplus \zeta_{2q-2}\F_q.$$
Fix a nontrivial additive character $\psi_0$ of $\F_q$, and denote by $\psi_1$ the nontrivial additive character of $\F_{q^2}$ given by
 $$\psi_1:=\psi_0 \circ \Trace_{\F_{q^2}/\F_q}.$$
Then the image $V^\star$ of $P(\infty)$ is itself isomorphic to $V$, and the representation of $P(\infty)$ is the direct sum of the characters
$$\oplus_{\alpha \in \F_q}\psi_1(\alpha x) \bigoplus \oplus_{\beta \in \F_q^\times}\psi_1(\zeta_{2q-2}\beta x).$$
\end{lem}
\begin{proof}When $n=2q-1$, then $n-1=2(q-1)$. The field $\F_{q^2}$ contains the $2(q-1)$'th roots of unity. The group
$\mu_{2(q-1)}(\F_{q^2})$ contains the subgroup $ \mu_{q-1}(\F_{q^2})=\F_q^\times$ with index $2$, the other coset being
$\zeta_{2(q-1)}\F_q^\times$. Thus the $\F_p$ span of $\mu_{2(q-1)}(\F_{q^2})$ inside the additive group of $\F_{q^2}$ is indeed the asserted $V$. The characters $\psi_1(\alpha x)$, as $\alpha$ varies over $\F_q$, are each trivial on $ \zeta_{2q-2}\F_q$ (because $\Trace_{\F_{q^2}/\F_q}(\zeta_{2q-2})=0$) and give all the additive characters of $\F_q$ (on which $\Trace_{\F_{q^2}/\F_q}$ is simply the map $x \mapsto 2x$). The characters $ \psi_1(\zeta_{2q-2}\beta x)$, as $\beta$ varies over $\F_q$, are trivial on $\F_q$ (because $\Trace_{\F_{q^2}/\F_q}(\zeta_{2q-2})=0$)  and give all the characters of  $\zeta_{2q-2}\F_q$ (because $\zeta_{2(q-1)}^2$ lies in $\F_q^\times$). 
\end{proof}

\begin{cor}The image of $P(\infty)$ in the $I(\infty)$-representation attached to $$\sF(k,2q-1,\psi) \oplus \triv$$ is the direct sum 
$$V= \F_q \oplus \zeta_{2q-2}\F_q$$ 
acting through the representation $${\rm Reg}_{\F_q}\oplus {\rm Reg}_{\zeta_{2q-2}\F_q}.$$
\end{cor}

\section{Basic facts about the group $G_{geom}$ for $\sF(k,2q-1,\psi)$}
Recall that $G_{geom}$ is the Zariski closure in $\SO(2q-1,\overline{\Q_\ell})$ of the image of $\pi_1^{geom}:=\pi_1(\A^1/\overline{\F_p})$ in the representation attached to $\sF(k,2q-1,\psi)$. Thus $G_{geom}$ is an irreducible subgroup of $\SO(2q-1,\overline{\Q_\ell})$.
\begin{thm}We have the following two results.
\begin{itemize}
\item[(1)]$G_{geom}$ is normalized by an element of $\SO(2q-1,\overline{\Q_\ell})$ whose eigenvalues are all the roots of unity of order dividing $2q-1$ in $\overline{\Q_\ell}$.
\item[(2)]$G_{geom}$ contains a subgroup isomorphic to $\F_q \oplus \F_q$, acting through the virtual representation
$${\rm Reg}_{\rm first}\oplus {\rm Reg}_{\rm second} - \triv.$$
\end{itemize}
\end{thm}
\begin{proof}The local system $\sF(k,2q-1,\psi)$ is, geometrically, a multiplicative translate of the Kummer pullback $[2q-1]^\star \sH_{2q-1}$. Therefore $G_{geom}$ is a normal subgroup of the group $G_{geom}$ for $\sH_{2q-1}$,
so is normalized by any element of this possibly larger group. As already noted, local monodromy at $0$ for $\sH_{2q-1}$ is an element of the asserted type. This proves (1). Statement (2) is just a repeating of what was proved in the previous lemma.
\end{proof}

\section{The third moment of  $\sF(k,2q-1,\psi)$ and of  $\sG(k,2q-1,\psi)$}
Let us recall the general set up. We are given a lisse $\sG$ on a lisse, geometrically connected curve $C/k$. We suppose that $\sG$ is $\iota$-pure of weight zero, for an embedding $\iota$ of $\overline{\Q_\ell}$ into $\C$. We denote by $V$ the $\overline{\Q_\ell}$-representation given by $\sG$, and by $G_{geom}$ the Zariski closure in $\GL(V)$ of the image of $\pi_1^{geom}(C/k)$. For an integer $n \ge 1$, the $n$'th moment of $\sG$ is the dimension of the space of invariants
$$M_n(\sG) := \dim((V^{\otimes n})^{G_{geom}}).$$

Recall \cite[1.17.4]{Ka-MMP} that we have an archimedean limit formula for $M_n(\sG)$ as the lim sup over finite extensions $L/k$ of the sums
$$(1/\#L)\sum_{t \in C(L)}(\Trace(Frob_{t,L}|\sG) )^n,$$
which we call the empirical moments.

\begin{thm} \label{third moment}For the lisse sheaf $\sG(k,2q-1,\psi)$ on $\A^1/k$, we have
$$M_3(\sG(k,2q-1,\psi))=1.$$
\end{thm}
\begin{proof}Fix a finite extension $L/k$. For $t \in L$, we have
$$\Trace(Frob_{t,L}|\sG(k,2q-1,\psi))=(-1/A(L,2q-1,\psi_{L/k}))\sum_{x \in L}\psi_{L/k}(x^{2q-1}+tx)\chi_{2,L}(x),$$
with the twisting factor given explicitly as
$$A(L,2q-1,\psi_{L/k})=-\chi_{2,L}(-1)\sum_{x \in L^\times}\psi_{L/k}(x)\chi_{2,L}(x).$$
Write $g_L$ for the Gauss sum
$$g_L:=\sum_{x \in L^\times}\psi_{L/k}(x)\chi_{2,L}(x).$$
Then the empirical $M_3$ is the sum
$$(1/\#L)(\chi_{2,L}(-1)/g_L)^3\sum_{t \in L}\sum_{x,y,z \in L}(\psi_{L/k}(x^{2q-1}+y^{2q-1}+z^{2q-1}+t(x+y+z))\chi_{2,L}(xyz)=$$
$$=(\chi_{2,L}(-1)/g_L)^3\sum_{x,y,z \in L, x+y+z=0}\psi_{L/k}(x^{2q-1}+y^{2q-1}+z^{2q-1})\chi_{2,L}(xyz)=$$
$$=(\chi_{2,L}(-1)/g_L)^3\sum_{x,y \in L}\psi_{L/k}(x^{2q-1}+y^{2q-1}+(-x-y)^{2q-1})\chi_{2,L}(xy(-x-y)).$$
The key is now the following identity.
\begin{lem} In $ \F_q[x,y]$, we have the identity
$$x^{2q-1}+y^{2q-1}+(-x-y)^{2q-1}=xy(x+y)\prod_{\alpha \in \F_q \setminus \{0,-1\}}(x-\alpha y)^2.$$
If we write $q=p^f$, then collecting Galois-conjugate terms this is
$$xy(x+y)\prod_{h \in \sP_f}h(x,y)^2,$$
where $\sP_f$ is the set of irreducible $h(x,y) \in \F_p[x,y]$ which are homogeneous of degree dividing $f$, monic in $x$, other than $x$ or $x+y$.
\end{lem}
\begin{proof}Because $x^{2q-1}+y^{2q-1}+(-x-y)^{2q-1}$ is homogeneous of odd degree $2q-1$ and visibly divisible by $y$, it suffices to prove the inhomogenous identity, that in $\F_q[x]$ we have
$$x^{2q-1}+1-(x+1)^{2q-1}=x(x+1)\prod_{\alpha \in \F_q \setminus \{0,-1\}}(x-\alpha )^2.$$
The left side 
$$P(x):=x^{2q-1}+1-(x+1)^{2q-1}$$
has degree $2q$. 

So it suffices to show that for each $\alpha \in \F_q \setminus \{0,-1\}$,$P(x)$ is divisible by $(x-\alpha)^2$. The key point is that for $\beta \in \F_q$, we have
$$\beta^{2q-1}=\beta,$$
and for $\alpha \in \F_q^\times$ we have
$$\alpha^{2q-2}=1.$$
Thus for any $\beta \in \F_q$, we trivially have $P(\beta) = 0$. The derivative $P'(x)$ is equal to
$$P'(x)=-x^{2q-2} +(x+1)^{2q-2}.$$
So if both $\alpha$ and $\alpha +1$ lie in $\F_q^\times$, then $P'(\alpha) =-1 + 1=0.$
\end{proof}

With this identity in hand, we now return to the calculation of the empirical moment, which is now
$$(\chi_{2,L}(-1)/g_L)^3\sum_{x,y \in L}\psi_{L/k}(xy(x+y)\prod_{h \in \sP_f}h(x,y)^2)\chi_{2,L}(xy(-x-y)).$$
The set of $(x,y) \in \A^2(L)$ at which $\prod_{h \in \sP_f}h(x,y)=0$ has cardinality at most $(q-2)(\#L -1)$.
So the empirical sum differs from the modified empirical sum
$$(\chi_{2,L}(-1)/g_L)^3\sum_{x,y \in L}\psi_{L/k}(xy(x+y)\prod_{h \in \sP_f}h(x,y)^2)\chi_{2,L}(xy(-x-y)\prod_{h \in \sP_f}h(x,y)^2)$$
by a difference which is 
$$(\chi_{2,L}(-1)/g_L)^3({\rm a\ sum\ of\ at\  most \ }(q-2)(\#L -1)\ {\rm terms,\ each\ of\ absolute\ value\ }1).$$
So the difference in absolute value is at most $q/\sqrt(\#L)$, which tends to zero as $L$ grows (remember $q$ is fixed).
The modified empirical sum we now rewrite as
$$(\chi_{2,L}(-1)/g_L)^3\sum_{t \in L^\times}\psi_{L/k}(t)\chi_{2,L}(-t)N_L(t),$$
with $N_L(t)$ the number of $L$-points on the curve $\sC_t$ given by
$$\sC_t: xy(x+y)\prod_{h \in \sP_f}h(x,y)^2 =t.$$
Because $ xy(x+y)\prod_{h \in \sP_f}h(x,y)^2$ is homogeneous of degree $2q-1$ prime to $p$ and is not a $d$'th power for any $d \ge 2$, the curves $\sC_t$ are smooth and geometrically irreducible for all $t \neq 0$, cf. \cite[proof of 6.5]{Ka-PES}. Moreover, by the homogeneity, these curves are each geometrically isomorphic to $\sC_1$, indeed the family become constant after the tame Kummer pullback $[2q-1]^\star$. Thus for the structural map $\pi:\sC \rightarrow \G_m/\F_p$, $R^2\pi_!(\Q_\ell) = \Q_\ell(-1)$, $R^1\pi_!\Q_\ell$ is lisse of some rank $r$, tame at both $0$ and $\infty$, and mixed of weight $\le 1$, and all other $R^i\pi_!(\Q_\ell)=0$.

So our modified empirical moment is
$$(\chi_{2,L}(-1)/g_L)^3\sum_{t \in L^\times}\psi_{L/k}(t)\chi_{2,L}(-t)(\#L -\Trace(Frob_{t,L}|R^1\pi_!\Q_\ell )=$$
$$=(\chi_{2,L}(-1)/g_L)^3\sum_{t \in L^\times}\psi_{L/k}(t)\chi_{2,L}(-t)(\#L) $$
$$-(\chi_{2,L}(-1)/g_L)^3 \sum_{t \in L^\times}\Trace(Frob_{t,L}|\sL_{\chi(t)}\otimes \sL_{\chi_2(t)}\otimes R^1\pi_!\Q_\ell).$$
Remembering that $g_L^2 = \chi_{2,L}(-1)$, we see that the first sum is $\chi_{2,L}(-1)$. We now show that the second sum is $O(1/\sqrt{\#L})$, or equivalently that the sum
$$ \sum_{t \in L^\times}\Trace(Frob_{t,L}|\sL_{\psi}\otimes \sL_{\chi_2}\otimes R^1\pi_!\Q_\ell)$$
is $O(\#L)$. By the Lefschetz trace formula \cite{Gro-FL}, the second sum is 
$$\Trace(Frob_L|H^2_c(\G_m/\overline{\F_p}, \sL_{\psi}\otimes \sL_{\chi_2}\otimes R^1\pi_!\Q_\ell)$$
$$-\Trace(Frob_L|H^1_c(\G_m/\overline{\F_p}, \sL_{\psi}\otimes \sL_{\chi_2}\otimes R^1\pi_!\Q_\ell).$$
The $H^2_c$ group vanishes, because the coefficient sheaf is totally wild at $\infty$ (this because it is $\sL_\psi$ tensored with a lisse sheaf which is tame at $\infty$). The second sum is $O(\#L)$, by Deligne's fundamental estimate \cite[3.3.1]{De-Weil II} (because the coefficient sheaf of mixed of weight $\le 1$, so its $H^1_c$ is mixed of weight $\le 2$).

Thus the empirical moment is $\chi_{2,L}(-1)$ plus an error term which, as $L$ grows, is $O(1/\sqrt{\#L})$. So the lim sup is $1$, as asserted.
\end{proof}
\section{Exact determination of $G_{arith}$}
\begin{thm}Suppose known that $\sG(k,2q-1,\psi)$ has $G_{geom}=\alt(2q)$. Then its $G_{arith}$ is as asserted in Theorem 3.1, namely it is $\alt(2q)$ if $-1$ is a square in $k$, and is $\sym(2q)$ if $-1$ is not a square in $k$.
\end{thm}
\begin{proof}For $q > 3$, the outer automorphism group of $\alt(2q)$ has order $2$, induced by the conjugation action of $\sym(2q)$. Therefore the normalizer of $\alt(2q)$ in $\SO(2q-1)$ (viewed there by its deleted permutation representation) is the group $\sym(2q)$ (viewed in $\SO(2q-1)$ by (deleted permutation representation)$\otimes$sgn).  If $q=3$, the automoprhism group is slightly bigger but
the stabilizer of the character of the deleted permutation module is just $\sym(2q)$. [Indeed, either of the exotic automorphisms of $\alt(6)$ maps the cycle $(123)$ to an element which in $\sym(6)$ is conjugate to $(123)(456)$.
The element $(123)$ has trace $2$, whereas $(123)(456)$ has trace $-1$ (both viewed in $\SO(5)$ by the deleted permutation representation)].
Since we have a priori inclusions
$$G_{geom} = \alt(2q) \triangleleft G_{arith} \subset \SO(2q-1),$$
the only choices for $G_{arith}$ are $\alt(2q)$ or $\sym(2q)$.

Denoting by $V$ the representation of $G_{arith}$ given by $\sG(k,2q-1,\psi)$, the action of $G_{arith}$ on line
$$\L := (V^{\otimes 3})^{G_{geom}}$$
is a character of $G_{arith}/G_{geom}$. We claim that this character is the sign character  $\sgn$ of $G_{arith} \subset \sym(2q)$. To see this, we argue as follows.

For any $n \ge 3$, denoting by $V_{n}$ the deleted permutation representation of $\sym(n+1)$, one knows that 
$$(V_{n}^{\otimes 3})^{\sym(n+1)} =(V_{n}^{\otimes 3})^{\alt(n+1)}$$
is one dimensional, 
[Indeed, if $S^\lambda$ denotes the complex irreducible representation of $\sym({n+1})$ labeled by 
the partition $\lambda \vdash (n+1)$, then $V_n = S^{(n,1)}$ and $\sgn = S^{(1^{n+1})}$. An application of the Littlewood-Richardson
rule to 
$$S^\lambda \otimes {\mathrm {Ind}}^{\sym({n+1})}_{\sym(n}(S^{(n)}) = S^{\lambda} \oplus (S^\lambda \otimes V_n)$$ 
yields
$$V_n \otimes V_n = S^{(n+1)} \oplus S^{(n,1)} \oplus S^{(n-1,2)} \oplus S^{(n-1,1^2)}$$
see \cite[Exercise 4.19]{FH}. Further similar applications of the Littlewood-Richardson rule then show that 
$V_n \otimes V_n \otimes V_n$ contains the trivial representation $S^{(n+1)}$ once but does not contain $\sgn$.]
Hence that the action of $\sym(n+1)$ on $$((V_{n}\otimes \sgn)^{\otimes 3})^{\alt(n+1)}$$
is $\sgn^{3} =\sgn$. Taking $n=2q-1$, we get the claim.

Now apply Deligne's equidistribution theorem, in the form \cite[9.7.10]{Ka-Sar}. It tells us that if $G_{arith}/G_{geom}$ has order $2$ instead of $1$, then the Frobenii $Frob_{t,L}$ as $L$ runs over larger and larger extensions of $k$ of even (respectively odd) degree become equidistribued in the conjugacy classes of $G_{arith}$ lying in $G_{geom}$ (respectively lying in the other coset $G_{arith} \setminus G_{geom}$). Then as $L/k$ runs over finite extensions of {\bf odd} degree, the empirical third moment will be $-1 + O(1/\sqrt{\#L})$. On the other hand, if $G_{arith}=G_{geom}$, then every empirical moment will be $1 + O(1/\sqrt{\#L})$.

\section{Identifying the group}

In this section, we use the information obtained earlier to identify the group. We choose a field embedding
$\overline{\Q_\ell} \subset \C$, so that we may view $G:=G_{geom}$ as an algebraic group over $\C$.

So let $p$ be an odd prime with $q$ a power of $p$.
We start by assuming that $G$ is an irreducible, Zariski closed subgroup of 
$\SO(2q-1,\C)=\SO(V)$ such that $G$ contains 
 $Q$, an elementary abelian subgroup of order $q^2$.  Moreover, we assume that we may
write $Q=Q_1 \times Q_2$ so  that $V=V_0 \oplus V_1 \oplus V_2$
where $V_0$ is a trivial $Q$-module,  $V_0 \oplus V_i$ is the regular representation
for $Q_i$ and $Q_i$ acts trivially on the other summand.    Moreover, we assume
that $G$ is a quasi-$p$ group (in the sense that the subgroup generated by its $p$-elements is Zariski dense), cf. Lemma \ref{dense}.

\begin{lem}\label{tensor}  $V$ is tensor indecomposable for $Q_1$. More
precisely, $V \ne X_1 \otimes X_2$ where the $X_i$ are
$Q_1$-modules each of dimension $\ge 2$.
\end{lem}

\begin{proof} We argue by contradiction. Suppose $V = X_1 \otimes X_2$ with each $X_i$ of
(necessarily odd) dimensional $\ge 2$ .    Let $\chi_{X_i}$ be the character of $Q_1$ on $X_i$.
So $\chi_{X_1}= a_0\triv+ \sum a_{\chi} \chi$  and $\chi_{X_2} 
=b_0\triv + \sum b_{\chi} \chi$ where the $\chi$ are the
nontrivial characters of $Q_1$. 

We first reduce to the case when both $a_0,b_0$ are nonzero.
The multiplicity of the trivial character of $Q_1$ in $V$ is $q$, so we have
$$q  = a_0b_0 +\sum_{\chi}a_\chi b_{\overline{\chi}}.$$
So either $a_0b_0$ is nonzero, and we are done, or for some nontrivial $\chi$ we have $a_\chi  b_{\overline{\chi}}$ nonzero. In this latter case, replace $X_1$ by $X_1\otimes {\overline{\chi}}$ and $X_2$ by $X_2\otimes \chi $.

Since each nontrivial character $\chi$ of $Q_1$ occur exactly once
in $V$, for each such $\chi$ we have
$$1 =a_0b_\chi + a_\chi b_0 + \sum_{\rho \neq \chi} a_\rho b_{\chi \overline{\rho}}.$$
In particular we have the inequalities
$$a_0b_\chi \le 1,\ \  a_\chi b_0 \le 1.$$
Because $a_0, b_0$ are both nonzero, we infer that if $a_\chi \neq 0$, then $a_\chi = b_0=1$ (respectively
that if $b_\chi \neq 0$, then $a_0=b_\chi =1$). It cannot be the case that all $a_\chi$ vanish, otherwise
$X_1$ is the trivial module of dimension $>1$. This is impossible so long as $X_2$ is nontrivial, as each nontrivial character of $Q_1$ occurs in $V$ exactly once. But if all $a_\chi$ and all $b_\chi$ vanish, then $V$ is the trivial $Q_1$ module, which it is not. Therefore $a_0=1$ and, similarly, $b_0=1$, and all $a_\chi, b_\chi$ are either $0$ or $1$.
Now use again that the multiplicity of the trivial character of $Q_1$ in $V$ is $q$, so we have
$$q  = a_0b_0 +\sum_{\chi}a_\chi b_{\overline{\chi}}.$$
This is possible only if all $a_\chi$ and all $b_\chi$ are $1$. But then each $X_i$ has dimension $q$, which is impossible, as the product of their dimensions is $2q-1$.
\end{proof}

\begin{lem}  \label{primitive} The following statements hold for $G$.
\begin{enumerate}[\rm(1)]
\item $G$ preserves no nontrivial orthogonal decomposition of $V$. 
\item $V$ is not tensor induced for $G$.
\end{enumerate}
\end{lem}

\begin{proof}   We first prove (1).  We argue by contradiction. Suppose that  $V=W_1 \perp \ldots \perp W_r$
with $r > 1$. Because $G$ acts irreducibly, $G$ transitively permutes the $W_i$, and
all the $W_i$ have the same odd
dimension $d$ (because $2q-1 = rd$).
Since $r$ divides $2q-1$,  $\gcd(r,p)=1$, so the $p$-group $Q$ fixes at least one of the $W_i$, say $W_1$. 
Because $r >1$, there are other orbits of $Q$ on the set of blocks. Any of these has cardinality some power of $p$,
so the corresponding direct sum of $W_i$'s has odd dimension. As $2q-1$ is odd, there must be evenly many other orbits,
so at least three orbits in total. In each $Q$-stable odd-dimensional orthogonal space, $Q$ lies in a maximal torus of the corresponding $\SO$ group, so has a fixed line. Hence $\dim V^Q \ge 3$, contradiction.

We next show that
$V$ is not tensor induced.  We argue by contradiction. If $V$ is tensor induced,  
write $V = W \otimes \ldots \otimes W$
(with $f \ge 2$ tensor factors, $\dim W < \dim V$).   Then $Q_1$
must act transitively on the set of tensor factors (otherwise the representation
for $Q_1$ is tensor decomposable and the previous lemma gives
a contradiction).  

So by Jordan's theorem, there exists an element $y \in Q_1$ that acts fixed point freely
on the set of the $f$ tensor factors.   All such elements are conjugate in
$\GL(W) \wr \sym(f)$ and we have
$$ \chi_V(y) = (\dim W)^{f/p}.$$
[Indeed, after replacing $y$ by a $\GL(W) \wr \sym(f)$-conjugate, the situation is this. Each orbit of $y$ on the set of tensor factors has length $p$,
and $y$ acts on each corresponding $p$-fold self product of $W$, indexed by $\F_p$, by
mapping $\otimes_i w_i$ to  $\otimes_i w_{i+1}$. In terms of a basis $B:=\{e_j \}_{j=1,...,\dim W}$ of W, the only diagonal entries of the matrix of $y$ on this $W^{\otimes p}$
are given by the $\dim W$ vectors $e \otimes e ... \otimes e$ with $e \in B$.]  On the other hand, we have
$ \chi_V(y) = q-1$ for any nonzero element $y$ of $Q_1$.
Thus, if $d=\dim W$, we have
$d^{f/p} = q-1$.  Thus,   $\dim V = d^f = (q-1)^p > 2q-1$, 
a contradiction.

\end{proof}

\begin{cor}\label{normal}
Let $L \leq \SO(V)$ be any subgroup containing $G$ and let $1 \neq N \lhd L$. Then $N$ acts irreducibly on $V$.
\end{cor}

\begin{proof}
We argue by contradiction. Note that the conclusions of Lemmas \ref{tensor} and \ref{primitive} also hold for $L$.

\smallskip
(i) Because $N$ is normal in $L$,
$V$ is completely reducible for $N$. Let $V_1,\ldots,V_r$ be the distinct $N$-isomorphism
classes of $N$-irreducible submodules of $V$. Because $V$ is $L$-self dual, it is a fortiori $N$-self dual. Therefore the set of 
$V_i$ is stable by passage to the $N$-dual, $V_i \mapsto V_i^\star$.
The group $L$ acts transitively on the set of the 
$V_i$. Either every $V_i$ is $N$-self-dual, or none is (the $L$-conjugates of an $N$-self dual representation
are $N$-self dual).

When we
write $V$ as the direct sum of its the $N$-isotypic (``homogeneous'' in the
terminology of \cite[49.5]{C-R}) components,
$$V= W_1 \oplus \ldots \oplus W_r,$$
then for some integer $e \ge 1$ we have $N$-isomorphisms
$$W_i \cong e V_i := \ {\rm the\ direct\ sum\ of\ }e\ {\rm copies\ of\ }V_i.$$


If $r > 1$ and all the $W_i$ are self dual, then this is an orthogonal decompositon
(because for $i \neq j$, the inner product pairing of (any) $V_i$ with (any) $V_j$ is an $N$-homomorphism from $V_i$ to $V_j^\star \cong V_j$, so vanishes). This contradicts Lemma \ref{primitive}.

Suppose $r > 1$ and no $V_i$ is self-dual. Then the $V_i$ occur in pairs of duals. Therefore both $r$ and $\dim V$ are
even, again a contradiction.

\smallskip
(ii) We have shown that $r=1$ and $e >1$, i.e. $V \cong eV_1$. 
Now we apply Clifford's theorem, cf. \cite[Theorem 51.7]{C-R}. Thus $L$ preserves the $N$-isomorphism class of $V_1$, and
so we get an irreducible projective representation $L \mapsto \PGL(V_1) =\PSL(V_1)$, and $V$ as a projective
representation of $L$ is $V_1\otimes X$ with $L$ acting (projectively) irreducibly on $X$ through $L/N$, and $X$ of dimension $e$. Furthermore, 
the two factor sets associated to these two projective representations are inverses to each other (because the tensor product $V_1 \otimes X =V$ is a linear representation of $Q_1$).
Since  $e\dim V_1 = \dim V = 2q-1$, both $e$ and $n = \dim V_1$ are coprime to $p$.

We now claim that, restricted to $Q_1$, each of the tensor factors $V_1$ and $X$ lifts to a genuine linear
representation. Indeed, using the fact that $\PGL(n,\C) =\PSL(n,\C)$ and the short exact sequence
$$1 \rightarrow \mu_n \rightarrow \SL(n,\C) \rightarrow \PSL(n, \C) \rightarrow 1,$$
the obstruction for $(V_1)|_{Q_1}$, which is given by the first factor set restricted to $Q_1$, lies in the cohomology group $H^2(Q_1,\mu_n)$.
As $p \nmid n$ while $Q_1$ is a $p$-group, this cohomology group vanishes; and so the first 
factor set restricted to $Q_1$ is trivial, whence so is the second.
Thus the $Q_1$-module $V$ is tensor decomposable, contradicting Lemma \ref{tensor}.
\end{proof}

We next show that $G$ is finite.  It is convenient to use one more fact about 
$G$.   There is a subgroup $A$ (namely the  group $G_{geom}$ for the hypergeometric sheaf $\sH_{2q-1}$) of $\SO(V)$ such that $G$ is normal in $A$,
$A/G$ is cyclic of order dividing $2q-1$ and $A$ contains an element $x$ of 
order $2q-1$ with distinct eigenvalues on $V$.   

We also use the fact  that  $G$ has a nontrivial fixed space on $V \otimes V \otimes V$
(Theorem \ref{third moment}).

\begin{thm}\label{finite}$G$ is finite.
\end{thm}

\begin{proof} 
Suppose not. 
Let $N$ be any nontrivial normal (closed) subgroup of $G$. By Corollary \ref{normal}, $N$ is irreducible on $V$.

(i) Let $G^0$ be the identity component of $G$. We now show that $G^0$ is a simple algebraic group. Taking $N=G^0$, we have that $G^0$ acts irreducibly and hence is semisimple (as it lies in $\SO(V )$). Moreover, the center of $G^0$ is trivial (because it consists of scalars in  $\SO(V )$). Therefore if $G^0$ is not simple, it is the product of adjoint groups $L_j$, $1 \leq j \leq t$ (namely the adjoint forms of the factors of its universal cover), and $V$ is the (outer) tensor product $V = \otimes^t_{j=1}V_j$ of nontrivial irreducible $L_j$-modules $V_j$. 
By \cite[Corollary 2.7]{GT},  $G$ permutes these tensor factors $V_j$. This action is transitive, otherwise we contradict Lemma \ref{tensor}.
But this implies that $V$ is tensor induced for $G$, contradicting Lemma \ref{primitive}. Thus $G^0$ is a simple algebraic group.
 
(ii)  Because the subgroup of $G$ generated by its $p$-elements is Zariski dense, the finite group $G/G^0$ is generated by its $p$-elements. As $p$ is odd, it follows
 that either $G = G^0$ is a simple algebraic group or  $p=3$ and  
 $G^0 = D_4(\C)$. [In all other cases, the outer automorphism
 group, i.e., the automorphism group of the Dynkin diagram of $G^0$, has order at most $2$.]     Since $A/G$ has odd order, it follows
 that $A \le G$ as well unless $G^0=D_4(\C)$ and $3$ divides $2q-1$.
  
Suppose first that $A$ is connected and so a simple algebraic group.  Then it contains
a semisimple element $x$ acting with distinct eigenvalues.  This implies that a maximal torus
has all weight spaces of dimension at most $1$.  Moreover, the module is in the root lattice
(since it is odd dimensional and orthogonal). By a result of Howe \cite{howe} (see also 
\cite[Table]{Pa}), it follows if $G \neq \SO(V)$, then either $G=G_2(\C)$ with $\dim V =7$ or $G=\PGL_2(\C)$.  
If $\dim V = 7$, then $q=4$, a contradiction.  If $G=\PGL_2(\C)$,
then any finite abelian subgroup of odd order is cyclic and so $Q$ does not embed in $G$.

So $G=\SO(V)$.  However, $\SO(V)$ has no nonzero fixed points on $V \otimes V \otimes V$ and
this contradicts Theorem \ref{third moment}. 

Thus, it follows that $A$ is disconnected.
So the connected component is $D_4(\C)$ and this acts irreducibly on $V$.
If $D_4(\C)$ contains the element of order $2q-1$, then a maximal torus has all
weight space of dimension $1$ and again using \cite{howe}, we obtain a contradiction.
If not, then $3$ divides $2q-1$, whence $p \ge 5$ and $Q \le D_4(\C)$.  
Any elementary abelian $p$-subgroup of $D_4(\C)$ is contained in a torus and so again
we see that the connected component has all weight spaces of dimension at most $1$
and we obtain the final contradiction using \cite{howe}.   
\end{proof}

Let $F^*(X)$ denote the generalized Fitting subgroup of a finite group $X$ (so $X$ is almost
simple precisely when $F^*(X)$ is a non-abelian simple group).

   \begin{cor}\label{fstar}   $A$ and $G$ are  almost simple and $F^*(A)=F^*(G)$
acts irreducibly on $V$. 
\end{cor}

\begin{proof}    
Let $N$ be a minimal normal subgroup of $G$.   
By Corollary \ref{normal}, $N$ acts irreducibly, and so by Schur's lemma $C_A(N) = Z(N)=1$ as
$A < \SO(V)$ with $\dim V$ odd.   So $N$ is a direct product of non-abelian
simple groups.   
Arguing as in p. (i) of the proof of Theorem \ref{finite}, we see
that $N$ is non-abelian simple (otherwise the module $V$ would be tensor induced).  As $C_G(N) = 1$, 
we see that $N \lhd G \leq \Aut(N)$, and so $G$ is almost simple and $F^*(G) = N$.   

Now, as $G \lhd A$, $A$ normalizes $N$. Again since $C_A(N) = 1$ we have that $N \lhd A \leq \Aut(N)$, 
and so $A$ is almost simple and $F^*(A) = N$.
\end{proof}

We next observe that:

\begin{lem} $F^*(G)$ is not a sporadic simple group.
\end{lem}
\begin{proof}  Notice that both $G$ and $A$ are generated by elements of odd order ($p$-elements for $G$, 
these and elements of order $2q-1$ for $A$). On the other hand, we have $S \leq G \leq A \leq \Aut(S)$ for
$S = F^*(G)$. One knows \cite{ATLAS} that if $S$ is sporadic, then $|\Out(S)| \leq 2$. Therefore, if $S$ is a sporadic simple group, then $G=A=S$. The result now folllows easily from information in \cite{ATLAS}.   
Namely, we observe that if $q$ is an odd prime power with $q^2$ dividing $|G|$,
then $G$ has no irreducible representation of dimension $2q-1$.  
\end{proof}

We next consider the case $F^*(G)=\alt(n)$.   First note $\alt(5)$ contains
no noncyclic elementary abelian groups of odd order and so is ruled
out.  Since $2q-1$ is odd,  we see that if $G=\alt(n)$, then $A=G = \alt(n)$  (as
the outer automorphism group of $\alt(n)$ is a $2$-group).

\begin{thm}\label{alt}  Let $\Gamma=\alt(n)$ with $n \ge 6$.  Suppose that $x \in \Gamma$
has odd order and $V$ is an irreducible $\C[\Gamma]$-module such that
$x$ acts as a semisimple regular element on $V$.   Then   
one the following holds:
\begin{itemize}
\item[(1)]$V$ is the deleted permutation module of dimension $n-1$ {\em (i.e. the nontrivial irreducible constituent of
$\C_{\alt({n-1})}^{\alt(n)}$)}, and $x$ is either an $n$-cycle (for $n$ odd) or
a product of two disjoint cycles of coprime lengths (for $n$ even); or
\item[(2)]$n=8$,  $x$ has order $15$ and $\dim V=14$.
\end{itemize}
\end{thm}

\begin{proof}     First note that if $V$ is the deleted permutation
module of dimension $n-1$, an element with $3$ or more disjoint cycles has at least
a two dimensional fixed space on $V$.   Similarly, if $x$ has two disjoint cycles of lengths
which are not coprime, then $x$ has a two dimensional eigenspace on $V$.  

Next we observe that if $x$ is semisimple regular on 
$V$, then the order of $x$ is at least $\dim V$.   This proves the result 
for $6 \le n \le14$ by inspection of the odd order elements and dimensions
of the irreducible modules, aside from the case $n=8$ and $\dim V=14$
(note that $\alt(8)$ contains an element of order $15$).   Recall that $\alt(8) \cong 
\GL_4(2)$ and it acts $2$-transitively on the nonzero vectors.   The only irreducible
module of dimension $14$ 
is the irreducible summand of the permutation module of dimension $15$.   In this case
$x$ has a single orbit in the permutation representation and so $x$ is semisimple
regular on $V$.

Now assume that $n \ge 15$.  

Suppose first that $x$ has at most three nontrivial cycles. 
Then the order of $x$ is less than $(n/3)^3 = n^3/27$ and so 
$\dim V < n^3/27$.  Let $W$ be a complex irreducible $\sym(n)$-module whose restriction to
$\alt(n)$ contains $V|_{\alt(n)}$. Since $2 \leq \dim W < 2n^3/27$, it follows
by \cite[Result 3]{Ra} that  $W \cong S^\lambda$ or $S^\lambda \otimes {\sgn}$, where $S^{\lambda}$ is the Specht module
labeled by the partition $\lambda$ of $n$, with $\lambda = (n-1,1)$, $(n-2,2)$, or 
$(n-2,1,1)$. Restricting back to $\alt(n)$, we see that $V|_{\alt(n)} = S^\lambda|_{\alt(n)}$.

Note that 
$$\dim S^{(n-2,1,1)}= (n-1)(n-2)/2,\ \  \dim S^{(n-2,2)} = n(n-3)/2.$$  
 It is straightforward to see that the dimension
of the fixed space of $x$ on either of these modules is at least two dimensional,
a contradiction. Hence $\lambda = (n-1,1)$ and $V|_{\alt(n)}$ is the deleted permutation module of dimension $n-1$.

We now induct on $n$.   The base case $n  \le 14$ has already done.
We may assume that $x$ has at least four nontrivial cycles (each of odd length, as $x$ has odd order). 
View $x \in J :=\alt(a) \times \alt(b)$. where the projection into $\alt(b)$
is a $b$-cycle and so the projection into $\alt(a)$ is a product of at least three
disjoint cycles.   Thus, $a \ge 9$.   Let  $W$ be an irreducible
$J$-submodule of $V$ with $\alt(a)$ acting nontrivially.  So  $W=W_1 \otimes W_2$
with $W_1$ an irreducible $\alt(a)$-module.  
Then $x$ must be multiplicity
free on each $W_i$ and by induction $x$ can have at most two cycles in
$\alt(a)$, a contradiction.   
\end{proof}

Note that the previous result does fail for $n=5$.   $\alt(5)$ has a $5$-dimensional representation
in which an element of order $5$ has all eigenvalues occurring once.  
Thus if $G = \alt(n)$, we see that $n=2q$ and $V$ is the deleted permutation module.

\begin{cor}If $G=G_{geom}$ is an alternating group $\alt(n)$ for some $n$, then $n=2q$.
\end{cor}

\bigskip
Finally, we consider the case where $G$ is an almost simple finite group of Lie type, defined over 
$\F_s$, where $s = s_0^f$ is a power of a prime $s_0$. Let us denote
$$S := F^*(G) = F^*(A).$$
 Recall that $S$ is simple, irreducible on
$V$,  and $Z(S) = 1$ by Corollary \ref{fstar}. 
We will freely use information on character tables of simple groups available in
\cite{ATLAS,GAP}, as well as degrees of complex irreducible characters of various quasisimple groups of 
Lie type available in \cite{Lu}. Finally, we will also use bounds on the smallest degree $d(S)$ of nontrivial 
complex irreducible representations of $S$ as listed in \cite[Table 1]{T}.

\begin{thm}\label{cross}
Suppose $s_0 \neq p$. Then $S \cong \alt(m)$ with $m \in \{5,6,8\}$.
\end{thm}

\begin{proof}
(i) Assume the contrary. We will exploit the existence of the subgroup $Q \leq G$. Recall that the {\it $p$-rank} $m_p(G)$ is the largest rank of 
elementary abelian $p$-subgroups of $G$. Furthermore, 
\begin{equation}\label{aut}
  \Aut(S) \cong \Innd(S) \rtimes \Phi_S\Gamma_S,
\end{equation}  
where $\Innd(S)$ is the subgroup of inner-diagonal automorphisms of $S$, $\Phi_S$ is a subgroup of field automorphisms of 
$S$ and $\Gamma_S$ is a subgroup of graph automorphisms of $S$, as defined in \cite[Theorem 2.5.12]{GLS}. As $F^*(G)=S$, we can embed 
$G$ in $\Aut(S)$. 
Now, given an elementary abelian $p$-subgroup $P < G$ of rank $m_p(G)$, we can define a normal series 
$$1 \leq P_1 \leq P_2 \leq P,$$
where $P_1 = P \cap \Innd(S)$ and $P_2 = P \cap (\Innd(S) \rtimes \Phi_S)$. As $\Phi_S$ is cyclic and $P$ is elementary 
abelian, $P_2/P_1$ has order $1$ or $p$. Set $e = 1$ if $S \cong P\Omega^+_8(s)$ and $p = 3$, and $e=0$ otherwise. Then 
$|P/P_2| \leq p^e$.  

Next we bound $|P_1|$ when $S$ is not a Suzuki-Ree group. 
Let $\Phi_j(t)$ denote the $j^{\mathrm {th}}$ cyclotomic polynomial in the variable $t$, and let 
$m$ denote the multiplicative order of $s$ modulo $p$, so that $p|\Phi_m(s)$. Note that
we can find a simple algebraic group $\sG$ of adjoint type defined over $\overline\F_s$ and a Frobenius endomorphism
$F:\sG \to \sG$ such that $\Innd(S) \cong \sG^F$. Letting $r$ denote the rank of $\sG$, then one can find 
$r$ positive integers $k_1, \ldots ,k_r$ and $\eps_1, \ldots ,\eps_r = \pm 1$ such that
$$|\Innd(S)| = s^N\prod_{j \geq 1}\Phi_j(s)^{n_j} = s^N\prod^r_{i=1}(s^{k_i}-\eps_i)$$
for suitable integers $N$, $n_j$. Then, according to \cite[Theorem 4.10.3(b)]{GLS}, $|P_1| \leq p^{n_m}$. 
Let $\varphi(\cdot)$ denote the Euler function, so that $\deg(\Phi_m) = \varphi(m)$. Inspecting the 
integers $k_1, \ldots,k_r$, one sees that $n_m \leq r/\varphi(m)$. It follows that
$$|P_1| \leq \Phi_m(s)^{n_m} \leq ((s+1)^{\varphi(m)})^{r/\varphi(m)} \leq (s+1)^r.$$
In fact, one can verify that this bound on $|P_1|$ also holds for Suzuki-Ree groups.
Putting all the above estimates together, we obtain that
\begin{equation}\label{for-q}
 q^2 = |Q| \leq |P| \leq (s+1)^{r+1+e}.
\end{equation}   
We will show that this upper bound on $q$ contradicts the lower bound
\begin{equation}\label{for-d}
2q-1 = \dim V \geq d(S)
\end{equation}
in most of the cases. Let $f^*$ denote the odd part of the integer $f$.

\smallskip
(ii) First we handle the case when $S$ is of type $D_4$ or $\tw3 D_4$. Here, $q \leq (s+1)^3$ by \eqref{for-q}. On 
the other hand, $d(S) \geq s(s^4-s^2+1)$, contradicting \eqref{for-d} if $s \geq 3$. If $s =2$, then 
$\Phi_S\Gamma_S = C_3$, and so instead of \eqref{for-q} we now have that $q^2 \leq 3^5$, whence 
$q \leq 13$, $2q-1 \leq 25 < d(S)$, again a contradiction.

From now on we may assume $e=0$.

Next we consider the case $S = \PSL_2(s)$. Then $\Out(S) = C_{\gcd(2,s-1)} \times C_f$, and $m_p(S) \leq 1$. 
It follows that $Q$ is not contained in $S$ but in $S \rtimes C_f$ and $3 \leq p|f^*$, and furthermore
$q^2 = |Q| \leq (s+1)f^*$. As $d(S) \geq (s-1)/2$, \eqref{for-d} now implies that 
$$s+1 = s_0^f+1 \leq 16f^*,$$
a contradiction if $s_0 \geq 5$, or $s_0= 3$ and $f \geq 5$, or $s_0 = 2$ and $f \geq 7$. If $s_0=3$ and $f \leq 4$, then
$f^* = 3 = f = p$, forcing $p=s_0$, a contraction. Suppose $s_0=2$ and $f \leq 6$. If $p = 5$, then $f^* = 5$ and 
$m_p(G) = 1$, ruling out the existence of $Q$. If $p=3$, then $f = 3,6$, whence $q^2 \leq 9$ and $2q-1 \leq 5 < d(S)$.

Suppose that $S = \tw2 B_2(s)$ or $\tw2 G_2(s)$ with $s \geq 8$. Since $m_p(S) \leq 1$, we see that 
$q^2 \leq (s+1)f$, contradicting \eqref{for-d} as $d(S) \geq (s-1)\sqrt{s/2}$. 

Now we consider the remaining cases with $r = 2$. Then $q \leq (s+1)^{3/2}$ by \eqref{for-q}. This contradicts \eqref{for-d}
if $S = G_2(s)$ (and $s \geq 3$), as $d(S) \geq s^3-1$. Similarly, $S \not\cong \PSL_3(s)$ with $s \geq 5$ and 
$S \not\cong \PSU_3(s)$ with $s \geq 8$. If $S = \PSp_4(s)$, then the case $2 \nmid s \geq 19$ is ruled out since
$d(S) \geq (s^2-1)/2$, and similarly the case $2|s \geq 8$ is ruled out since $d(S) = s(s-1)^2/2$. In the remaining cases,
$\Phi_S\Gamma_S$ is a $2$-group, and so $Q \leq S$, $q^2 \leq (s+1)^2$, $q \leq s+1$. Now $\PSL_3(s)$ and $\PSU_3(s)$ 
with $s \geq 4$ are ruled out by \eqref{for-d}, and the same for $\PSp_4(s)$ with $s \geq 4$. Note that when $s=3$, 
$q \geq 4$ and so $\gcd(q,2s) \neq 1$, a contradiction. If $S = \SL_3(2)$, then $q=3$ and $S$ has no irreducible character of
degree $2q-1$. Finally, $\Sp_4(2)' \cong \alt(6)$.

Next we handle the groups with $r=3$. Here $q \leq (s+1)^2$ by \eqref{for-q}. Then \eqref{for-d} implies that $s \leq 5$. In this case,
$\Out(S)$ is a $2$-group, and so $Q \leq S$ and $q \leq (s+1)^{3/2}$ by \eqref{for-q}. Using \eqref{for-d}, we see that $s \leq 3$. The 
remaining groups $S$ cannot occur, since $S$ does not have a real-valued complex irreducible character of degree $2q-1$.

\smallskip
(iii) From now we may assume that $r \geq 4$ (and $S$ is not of type $D_4$ or $\tw3 D_4$). First we consider the case $s=2$.
If $S = \SL_n(2)$ with $n \geq 5$, then since $\Out(S) = C_2$, the arguments in (i) show that $q^2 \leq 3^{n-1}$. This contradicts
\eqref{for-d}, since $d(S) = 2^n-2$. Suppose $S = \SU_n(2)$ with $n \geq 7$. Note by \cite[Theorem 4.1]{TZ} 
that the first three nontrivial irreducible characters of $S$ are Weil characters and either non-real-valued or of even degree, and the next
characters have degree at least $(2^n-1)(2^{n-1}-4)/9$. Hence
\eqref{for-d} can be improved to 
$$2q-1 \geq (2^n-1)(2^{n-1}-4)/9,$$
which is impossible since $q^2 \leq 3^n$ by \eqref{for-q}. If $S = \PSU_n(2)$ with $n = 5,6$, then $q^2 \leq 3^6$, and 
$S$ has no nontrivial real-valued irreducible character of odd degree $\leq 2q-1 \leq 53$. If $S = \tw2 F_4(2)'$ or $F_4(2)$, then $q^2 \leq 3^5$, 
$q \leq 13$, and $S$ has no nontrivial real-valued irreducible character of odd degree $\leq 2q-1 \leq 25$. 

Suppose $S = \Sp_{2n}(2)$ or $\Omega^\pm_{2n}(2)$. Then $\Out(S)$ is a $2$-group (recall $S$ is not of type $D_4$), and so 
$q^2 \leq 3^n$. On the other hand, $d(S) \geq (2^n-1)(2^{n-1}-2)/3$, contradicting \eqref{for-d}. Finally, if $S$ is of type $E_8$, $E_7$,
$E_6$, or $\tw2 E_6$, then $q^2 \leq 3^8$ whereas $d(S) > 2^{10}$, again contradicting \eqref{for-d}.

\smallskip
(iv) Suppose that $S = \PSp_{2n}(s)$ with $n \geq 4$ and $2 \nmid s \geq 3$. Then by \eqref{for-q} and \eqref{for-d} we have
$$(s^n-1)/2 \leq 2q-1 \leq 2(s+1)^{(n+1)/2}-1,$$
implying $n \leq 5$ and $s = 3$. In this case, inspecting the order of $\PSp_{10}(3)$ we see that $q^2 \leq 121$, and so
$2q-1 \leq 21 < d(S)$, a contradiction.

Next suppose that $S = \PSU_n(3)$ with $n \geq 5$. Then $q \leq 2^n$ and $d(S) \geq (3^n-3)/4$, and so 
\eqref{for-d} implies that $n = 5$.  In this case, inspecting the order of $\SU_5(3)$ we see that $q^2 \leq 61$, and so
$2q-1 \leq 13 < d(S)$, again a contradiction. 

Now we may assume that $r \geq 4$, $s \geq 3$, $S \not\cong \PSp_{2n}(s)$ if $2 \nmid s$, and moreover $s \geq 4$ if $S \cong \PSU_n(s)$.
Then one can check that $d(S) \geq s^r \cdot (51/64)$ (with equality attained exactly when $S \cong \PSU_5(4)$). Hence 
\eqref{for-q} and \eqref{for-d} imply that
$$(51/64)^2 \cdot s^{2r} \leq d(S)^2 \leq (2q-1)^2 < 4(s+1)^{r+1} \leq 4 \cdot (4s/3)^{r+1},$$
and so
$$(3s/4)^{r-1} < 4 \cdot (64/51)^2 \cdot (4/3)^2,$$
which is impossible for $r \geq 4$.
\end{proof}

\begin{thm}\label{defi}
Suppose $s_0 = p$. Then $S \cong \alt(6)$.
\end{thm}

\begin{proof}
(i) Assume the contrary. We now exploit the existence of the element $x \in A$ of order $2q-1$ which has simple spectrum on $V$.
As before, we can embed $A$ in $\Aut(S)$ and again use the decomposition \eqref{aut}. Let 
$\langle y \rangle = \langle x \rangle \cap \Innd(S)$. We also view $S = \sG^F$ for some Frobenius endomorphism 
$F:\sG \to \sG$ of a simple algebraic group $\sG$ of adjoint type, defined over $\overline\F_p$. Note that $y$ is an $F$-stable 
semisimple element in $\sG$, hence it is contained in an $F$-stable maximal torus $\sT$ of $\sG$ by \cite[Corollary 3.16]{DM}.
It follows that $|y| \leq |\sT^F| \leq (s+1)^r$, if $r$ is the rank of $\sG$. Set $e = 3$ if $S$ is of type $D_4$ or $\tw3 D_4$, and 
$e = 1$ otherwise. Then the decomposition \eqref{aut} shows that 
$$|x|/|y| \leq ef^*,$$
where $f^*$ denotes the odd part of $f$ as before (and $s = p^f$). We have thus shown that
\begin{equation}\label{order}
   2q-1 = |x| \leq (s+1)^ref^*.
\end{equation}
We will frequently use the following remark:
\begin{equation}\label{for-f}
  \mbox{Either } f=1 \mbox{ and }s \geq 3f^*, \mbox{ or }s \geq 9f^*.
\end{equation}
We will show that in most of the cases \eqref{order} contradicts \eqref{for-d}. First we handle the case $S$ is of type 
$D_4$ or $\tw3 D_4$, whence $d(S) \geq s(s^4-s^2+1)$. Hence \eqref{order} and \eqref{for-f} imply that 
$$s(s^4-s^2+1) \leq 2q-1 \leq s(s+1)^4/3$$
if $f>1$, a contradiction. If $f=1$, then since $2q-1 = \dim V$ is coprime to $2s$, we see by \cite{Lu} that 
$$2q-1 > s^7/2 > 3(s+1)^4,$$
contradicting \eqref{order}.  

\smallskip
(ii) From now on we may assume that $e=1$. Next we rule out the case where $V|_S$ is a Weil module of
$S \in \{\PSL_n(s),\PSU_n(s)\}$ with $n \geq 3$, or $S = \PSp_{2n}(s)$ with $n \geq 2$. Indeed, in this case, if 
$S = \PSL_n(s)$ then  
$$\dim V = (s^n-s)/(s-1),~~(s^n-1)/(s-1)$$
is congruent to $0$ or $1$ modulo $p$ and so cannot be equal to $2q-1$. Similarly, if $S = \PSU_n(s)$,
then $V|_S$ can be a Weil module of dimension $2q-1$ only when $2|n$ and $\dim V = (s^n-1)/(s+1)$. In this case,
$$q = (2q-1)_p = ((s^n+s)/(s+1))_p = s$$
(where $N_p$ denotes the $p$-part of the integer $N$),
and so $2s-1 = (s^n+s)/(s+1)$, a contradiction. Likewise, if $S = \PSp_{2n}(s)$,
then $V|_S$ can be a Weil module of dimension $2q-1$ only when $p=3$ and $\dim V = (s^n+1)/2$. In this case,
$$s^n = (2\dim V-1)_p = (4q-3)_p,$$
and so $q = 3$ and $n = 2$. One can show that $\PSp_4(3)$ does possess a complex irreducible module of dimension
$2q-1=5$, with an element $x$ of order $5$ with simple spectrum on $V$ and a subgroup $Q \cong C_3^2$ with desired 
prescribed action on $V$; however, any such module is not self-dual.
Henceforth, for the aforementioned possibilities for $S$ we may assume that $\dim V \geq d_2(S)$, the next degree after the degree of 
Weil characters. Note that $d_2(S)$ for these simple groups $S$ is determined in Theorems 3.1, 4.1, and 5.2  of \cite{TZ}.

\smallskip
(iii) Suppose $S = \PSL_2(s)$; in particular, $s \neq 9$. Assume $f \geq 4$. As $\Out(S) = C_{2,s-1} \times C_f$, we see that 
$q^2 \leq sf_p < s^2/20$, whereas $2q-1 \geq d(S) \geq (s-1)/2$, a contradiction. If $f \leq 3$ but $f_p > 1$, then $f=p=3$, $s= 3^3$, 
$q^2 \leq sf = 3^4$, forcing $q=9$. But then $S = \PSL_2(27)$ has no irreducible character of degree $2q-1=17$. Thus $f_p = 1$,
$q^2 \leq s$, and so \eqref{for-d} implies that $s \leq 17$. As $s \neq 9$, we see that $m_p(G) = m_p(S) = 1$, contradicting the existence
of $Q$.

Next we consider the case $S = \PSL_3(s)$ or $\PSU_3(s)$. If $f > 1$, then \eqref{for-d}--\eqref{for-f} imply
$$(s-1)(s^2-s+1)/3 \leq d_2(S) \leq 2q-1 \leq (s+1)^2 s/9,$$
which is impossible. Thus $f=1$, whence 
$$(s-1)(s^2-s+1)/3 \leq d_2(S) \leq 2q-1 \leq (s+1)^2,$$
yielding $s \leq 5$. But if $s = 3$ or $5$, then any nontrivial $\chi \in \Irr(S)$ of odd degree coprime to $s$ is a Weil character, which has been 
ruled out in (ii).

Suppose now that $S = \PSL_4(s)$ or $\PSU_4(s)$. For $s \geq 5$ we have 
$$(s-1)(s^3-1)/2 \leq d_2(S) \leq 2q-1 \leq (s+1)^3 s/3,$$
which is possible only when $s \leq 11$. Thus $3 \leq s \leq 11$, whence $f^*=1$, and so 
$$(s-1)(s^3-1)/2 \leq d_2(S) \leq 2q-1 \leq (s+1)^3,$$
leading to $s = 3$. If $s = 3$, then any odd-order element in $G$ has order $\leq 13$, whereas $d(S) = 21$, contradicting \eqref{for-d}. 

To finish off type $A$, assume now that $S = \PSL_n(s)$ or $\PSU_n(s)$ with $n \geq 5$. Then \eqref{for-d}--\eqref{for-f} imply 
$$\frac{(s^n+1)(s^{n-1}-s^2)}{(s+1)(s^2-1)} \leq d_2(S) \leq 2q-1 \leq (s+1)^{n-1}s/3,$$
whence 
$$s^{2n-3} < (s+1)^ns/3 < s^{51n/40}$$
(because $(s+1)/s \leq 4/3 < 3^{11/40}$), a contradiction as $n \geq 5$.

\smallskip
(iv) Suppose $S = P\Omega^\pm_{2n}(s)$ with $n \geq 4$.
For $n \geq 5$ we get that
$$\frac{(s^n-1)(s^{n-1}-s)}{s^2-1} \leq d(S) \leq 2q-1 \leq (s+1)^nf \leq (s+1)^ns/3,$$
whence
$$s^{2n-3.1} < (s+1)^ns/3 < s^{51n/40},$$
a contradiction. If $n = 4$, then $S = P\Omega^-_8(s)$. In this case, since $2q-1$ is coprime to $2s$, \cite{Lu} implies that
$$2q-1 \geq (s^4+1)(s^2-s+1)/2 > (s+1)^4s/3,$$
again a contradiction.

Suppose $S = \PSp_{2n}(s)$ with $n \geq 2$ or $\Omega_{2n+1}(s)$ with $n \geq 3$.
Using the bound $2q-1 \geq d_2(S)$ for $S = \PSp_{2n}(s)$ and $2q-1 \geq d(S)$ otherwise, we get for $n \geq 3$ that
$$\frac{(s^n-1)(s^n-s)}{s^2-1} \leq 2q-1 \leq (s+1)^nf \leq (s+1)^ns/3,$$
whence
$$s^{2n-2.1} < (s+1)^ns/3 < s^{51n/40},$$
a contradiction. If $n = 2$, then $S = \PSp_4(s)$, and we have 
$$s(s-1)^2 \leq 2q-1 \leq (s+1)^s/3,$$
forcing $q \leq 9$. If $5 \leq q \leq 9$, then since the degree $2q-1 = \dim V$ is coprime to $2s$, we again get
$2q-1 > 300 \geq (s+1)^2s/3$. Finally, $\PSp_4(3)$ has no nontrivial non-Weil character of degree coprime to $6$. 

\smallskip
(v) If $S$ is of type $E_6$, $\tw2 E_6$, $E_7$, or $E_8$, then
$$(s^5+s)(s^6-s^3+1) \leq  d(S) \leq 2q-1 \leq (s+1)^8f \leq (s+1)^8s/3,$$
a contradiction. Similarly, if $S = F_4(s)$, then
$$s^8-s^4+1 = d(S) \leq 2q-1 \leq (s+1)^4s/3,$$
which is impossible. Likewise, if $S = G_2(s)$ with $s \geq 5$, then
$$s^3-1 \leq d(S) \leq 2q-1 \leq (s+1)^2s/3,$$
again a contradiction. Next, if $S = G_2(3)$, then $2q-1 \leq 16$ cannot be a degree of an irreducible character of $S$.
Finally, if $S = \tw2 G_2(s)$, then 
$$s^2-s+1 = d(S) \leq 2q-1 \leq (s+1)f \leq (s+1)s/3,$$
again a contradiction since $s \geq 27$. 

\end{proof}

Our proof is now concluded by applying Theorem \ref{alt}. 
\end{proof}


\begin{thebibliography}{99}

\bibitem[As]{asch} Aschbacher, M., Maximal subgroups of classical groups, 
 On the maximal subgroups of the finite classical groups,
  Invent. Math. 76 (1984),  469--514. 
\medskip
\bibitem[Atlas]{ATLAS} Conway, J. H.,  Curtis, R. T.,  Norton, S. P.,  Parker, R. A. and 
 Wilson, R. A.,  Atlas of Finite Groups, Clarendon Press, Oxford 1985.
\medskip
\bibitem[C-R]{C-R}Curtis, C.W. and Reiner, I., Representation Theory of Finite Groups and Associative Algebras,
Interscience Publishers, New York 1962, xiv+689 pp.
\medskip
\bibitem[De-Weil II]{De-Weil II}Deligne, P., La conjecture de Weil II, Pub. Math. I.H.E.S. 52 (1981), 313--428.
\medskip
\bibitem[DM]{DM}
  Digne, F. and Michel, J.,  Representations of Finite Groups of
Lie Type, London Mathematical Society Student Texts 21, 
Cambridge University Press, 1991.
\medskip
\bibitem[FH]{FH}
  W. Fulton and J. Harris, Representation Theory, Springer-Verlag, New
York, 1991.
\medskip
\bibitem[GAP]{GAP}
The GAP group, {\sf GAP} - groups, algorithms, and
programming, Version 4.4, 
2004, \url{http://www.gap-system.org}.
\medskip
\bibitem[GLS]{GLS}
Gorenstein, D., Lyons, R., and Solomon, R.M.,
  The Classification of the Finite Simple Groups, Number 3.
  Part {I}. {C}hapter {A}, volume~40 of Mathematical Surveys and
  Monographs, American Mathematical Society, Providence, RI, 1998.
\medskip
\bibitem[Gross]{Gross}Gross, B. H., Rigid local systems on $\G_m$ with finite monodromy, Adv. Math. 224 (2010), no. 6, 2531--2543.
\medskip
\bibitem[Gro-FL]{Gro-FL} Grothendieck, A., Formule de Lefschetz et rationalit\'{e} des fonctions L, Seminaire Bourbaki 1964-65, Expos\'{e} 279, reprinted in: Dix Expos\'{e}s sur la cohomologie des sch\'{e}mas, North-Holland, 1968.
\medskip
\bibitem[GT]{GT} Guranick, R. M. and Tiep, P. H., Symmetric powers and a conjecture of Kollar and Larsen, Invent. Math. 174 (2008), 505--554.
\medskip
\bibitem[Ho]{howe} Howe, R.,  Another look at the local $\theta$-correspondence for an 
unramified dual pair. Festschrift in honor of I. I. Piatetski-Shapiro on the 
occasion of his sixtieth birthday, Part I (Ramat Aviv, 1989), 93--124, Israel Math. Conf. Proc., 2, Weizmann, Jerusalem, 1990.
\medskip
\bibitem[Ka-ESDE]{Ka-ESDE}Katz, N., Exponential Sums and 
Differential Equations, Annals of Mathematics Studies, 124. Princeton Univ. Press, Princeton, NJ, 1990. xii+430 pp.
\medskip
\bibitem[Ka-GKM]{Ka-GKM}Katz, N., Gauss Sums, Kloosterman Sums, and Monodromy Groups, Annals of Mathematics Studies, 116. Princeton Univ. Press, Princeton, NJ, 1988. ix+246 pp.
\medskip
\bibitem[Ka-MMP]{Ka-MMP}Katz, N., Moments, Monodromy, and Perversity,  Annals of Mathematics Studies, 159. Princeton University Press, Princeton, NJ, 2005. viii+475 pp.
\medskip
\bibitem[Ka-NG2]{Ka-NG2}Katz, N., Notes on $G_2$, determinants, and equidistribution, Finite Felds and their Applications 10 (2004), 221--269.
\medskip
\bibitem[Ka-PES]{Ka-PES}Katz, N., Perversity and exponential sums, in: Algebraic Number Theory - in honor of K. Iwasawa, Advanced Studies in Pure Mathematics 17, 1989, 209--259.
\medskip
\bibitem[Ka-RLS]{Ka-RLS}Katz, N., Rigid Local Systems, Annals of Mathematics Studies, 139. Princeton University Press, Princeton, NJ, 1996. viii+223 pp.
\medskip
\bibitem[Ka-RLSFM]{Ka-RLSFM}Katz, N., Rigid local systems on $\A^1$ with finite monodromy, preprint available at 
\url{www.math.princeton.edu/~nmk/gpconj87.pdf}.
\medskip
\bibitem[Ka-Sar]{Ka-Sar}Katz, N., and Sarnak, P., Random matrices, Frobenius eigenvalues, and monodromy, American Mathematical Society Colloquium Publications, 45. American Mathematical Society, Providence, RI, 1999. xii+419 pp.
\medskip
\bibitem[LS]{LS}Liebeck, M. W. and Seitz, G. M. On the subgroup structure of classical groups, Invent. Math. 134 (1998),   427--453.
\medskip
\bibitem[Lu]{Lu}
  L\"ubeck, F., Character degrees and their multiplicities for some groups of Lie type of rank $< 9$,
\url{http://www.math.rwth-aachen.de/~Frank.Luebeck/chev/DegMult/index.html}.
\medskip
\bibitem[Pa]{Pa}Panyushev, D. I.,  Weight multiplicity free representations, $g$-endomorphism algebras, and Dynkin polynomials,  J.  London Math. Soc.    69 (2004),  273--290.
 \medskip
\bibitem[Ra]{Ra} Rasala, R., On the minimal degrees of characters of $S_n$, J. Algebra 45 (1977), 132--181.
 \medskip
 \bibitem[Ray]{Ray}Raynaud, M. Rev\^{e}tements de la droite affine en caract\'{e}ristique p $>$ 0 et conjecture d'Abhyankar,
 Invent. Math. 116 (1994), no. 1-3, 425--462. 
 \medskip\bibitem[T]{T}
Tiep, P. H., Low dimensional representations of finite quasisimple groups, Proceedings of the London Math. Soc. Symposium ``Groups, Geometries, and Combinatorics'', Durham, July 2001, A. A. Ivanov et al eds., World Scientific, 2003, N. J. 277--294.  
\medskip
\bibitem[TZ]{TZ}
Tiep, P. H. and Zalesskii, A. E., Minimal characters of the finite classical
groups, Comm. Algebra 24 (1996), 2093--2167.
\end{thebibliography}
\end{document}